\PassOptionsToPackage{unicode}{hyperref}
\PassOptionsToPackage{hyphens}{url}
\documentclass[
]{article}
\usepackage{lmodern}
\usepackage{amssymb,amsmath}
\usepackage{ifxetex,ifluatex}
\ifnum 0\ifxetex 1\fi\ifluatex 1\fi=0 
  \usepackage[T1]{fontenc}
  \usepackage[utf8]{inputenc}
  \usepackage{textcomp} 
\else 
  \usepackage{unicode-math}
  \defaultfontfeatures{Scale=MatchLowercase}
  \defaultfontfeatures[\rmfamily]{Ligatures=TeX,Scale=1}
\fi
\IfFileExists{upquote.sty}{\usepackage{upquote}}{}
\IfFileExists{microtype.sty}{
  \usepackage[]{microtype}
  \UseMicrotypeSet[protrusion]{basicmath} 
}{}
\makeatletter
\@ifundefined{KOMAClassName}{
  \IfFileExists{parskip.sty}{%
    \usepackage{parskip}
  }{
    \setlength{\parindent}{0pt}
    \setlength{\parskip}{6pt plus 2pt minus 1pt}}
}{
  \KOMAoptions{parskip=half}}
\makeatother
\usepackage{xcolor}
\IfFileExists{xurl.sty}{\usepackage{xurl}}{} 
\IfFileExists{bookmark.sty}{\usepackage{bookmark}}{\usepackage{hyperref}}
\hypersetup{
  pdftitle={Tail inequalities for restricted classes of discrete random variables},
  pdfauthor={Mark Huber},
  hidelinks,
  pdfcreator={LaTeX via pandoc}}
\usepackage[margin=1in]{geometry}
\usepackage{graphicx,grffile}
\makeatletter
\def\maxwidth{\ifdim\Gin@nat@width>\linewidth\linewidth\else\Gin@nat@width\fi}
\def\maxheight{\ifdim\Gin@nat@height>\textheight\textheight\else\Gin@nat@height\fi}
\makeatother
\setkeys{Gin}{width=\maxwidth,height=\maxheight,keepaspectratio}
\makeatletter
\def\fps@figure{htbp}
\makeatother
\providecommand{\tightlist}{%
  \setlength{\itemsep}{0pt}\setlength{\parskip}{0pt}}
\usepackage{amsthm}

\newtheorem{lemma}{Lemma}
\newtheorem{thm}{Theorem}

\title{Tail inequalities for restricted classes of discrete random variables}
\author{Mark Huber}
\date{2021-01-09}

\begin{document}
\maketitle

\newcommand{\unifdist}{\textsf{Unif}}
\newcommand{\berndist}{\textsf{Bern}}
\newcommand{\sqrtfcn}{\operatorname{sqrt}}
\newcommand{\sign}{\operatorname{sgn}}
\newcommand{\prob}{\mathbb{P}}
\newcommand{\mean}{\mathbb{E}}
\newcommand{\var}{\mathbb{V}}
\newcommand{\ind}{\mathbb{I}}

\hypertarget{abstract}{%
\subsection*{Abstract}\label{abstract}}
\addcontentsline{toc}{subsection}{Abstract}

Let \(X\) be an integrable discrete random variable over
\(\{0, 1, 2, \ldots\}\) with
\(\mathbb{P}(X = i + 1) \leq \mathbb{P}(X = i)\) for all \(i\). Then for
any integer \(a \geq 1\), \[
\mathbb{P}(X \leq a) \leq \frac{\mathbb{E}[X]}{2a - 1}.
\] Let \(W\) be an discrete random variable over
\(\{\ldots, -2, -1, 0, 1, 2, \ldots\}\) with finite second moment where
the \(\mathbb{P}(W = i)\) values are unimodal. Then \[
\mathbb{P}(|W - \mathbb{E}[W]| \geq a) \leq \frac{\mathbb{V}(W) + 1 / 12}{2(a - 1 / 2)^2}.
\]

\hypertarget{introduction}{%
\section{Introduction}\label{introduction}}

Markov's and Chebyshev's inequalities (Markov 1884; Tchebichef 1867) are
two widely used bounds on the tail probabilities of random variables,
owing to the limited assumptions needed for application. They have many
uses in the construction and analysis of a wide number of randomized
algorithms. See (Motwani and Raghavan 1995) and (Williamson and Shmoys
2011) for details.

In this work, improvements upon the general bounds under restrictions on
the density of the random variable are considered. These bounds are
almost half as large as the original, leading to improvements in both
algorithmic design and analysis.

\hypertarget{a-typical-problem}{%
\subsection{A typical problem}\label{a-typical-problem}}

Consider the following problem.

\begin{quote}
For \(U\) uniform over \(\{0, 1, 2, \ldots, 10\}\), bound
\(\mathbb{P}(U \geq 9)\) using Markov's inequality and Chebyshev's
inequality.
\end{quote}

The true answer is \(\mathbb{P}(U \geq 9) = 2 / 11= 0.181818\ldots\).
How well do the classical tail inequalities do?

Markov's inequality states that for an integrable random variable \(X\)
and \(a \geq 0\), \[
\mathbb{P}(|X| \geq a) \leq \mathbb{E}[|X|] / a.
\] Hence Markov's inequality gives for the problem that
\(\mathbb{P}(U \geq 9) \leq 5 / 9\), or about 55.55\%.

Chebyshev's inequality states that for a random variable with finite
second moment, \[
\mathbb{P}(|X - \mathbb{E}(X)| \geq a) \leq \mathbb{V}(X) / a^2.
\] Hence Chebyshev's inequality gives for the problem that
\(\mathbb{P}(U \geq 9) \leq \mathbb{P}(|U - 5| \geq 4) \leq [(11^2 - 1)/12] / 4^2\),
which is 62.5\%.

Both tail bounds are abysmal. But why? One answer lies in the fact that
Chebyshev bounds both tails rather than just one. By the symmetry of a
uniform distribution,
\(2 \mathbb{P}(U - 5 \geq 4) = \mathbb{P}(|U - 5| \geq 4)\), which
brings the Chebyshev bound down to 31.25\%. But how to improve Markov?

In (Huber 2019), it was shown that common tail inequalities such as
Markov's inequality and Chebyshev's inequality could be cut in half
using smoothing. This smoothing technique works for any type of random
variable: continuous, discrete, or a mixture of the two. Write
\(U \sim \textsf{Unif}([a, b])\) to mean that \(U\) is a continuous
uniform over the interval \([a, b]\).

\begin{lemma}

For \(W \sim \textsf{Unif}([-a, a])\) and any integrable random variable
\(X\), \[
\mathbb{P}(|X| + W \geq a) \leq \frac{1}{2} \frac{\mathbb{E}[|X|]}{a}.
\]

\end{lemma}

A similar result holds for Chebyshev's inequality.

\begin{lemma}

For \(W \sim \textsf{Unif}([-a / 2, a / 2])\) and a random variable
\(X\) with finite second moment, \[
\mathbb{P}\left(\left||X| + W \right| \geq a \right) \leq \frac{1}{2} \frac{\mathbb{V}(X)}{a^2}.
\]

\end{lemma}

While this smoothing works for all types of random variables, it is
often the case that this addition cannot be used for a particular
application. When this happens, for continuous random variables, it was
also possible to bound the tail probabilities of the original \(X\) via
the tail probabilities of the smoothed variable.

\begin{lemma}

Let \(X\) be a nonnegative continuous random variable with decreasing
density, \(0 \leq c \leq a\), and let \(W\) be a random variable
independent of \(X\) such that \(W \sim \textsf{Unif}([-c, c]).\) Then
\(\mathbb{P}(X + W \geq a) \geq \mathbb{P}(X \geq a)\).

\end{lemma}

For example, the density of an exponential random variable \(T\) with
rate \(2\) is \(f_T(t) = 2 \exp(-2t) \mathbb{I}(t \geq 0)\). Here
\(\mathbb{I}\) is the indicator function that is 1 when the argument is
true and 0 when the argument is false. Then
\(\mathbb{P}(T \geq 1.5) = \exp(-2\cdot 1.5) = 0.04978\ldots\). Since
\(T\) is both nonnegative and the density is decreasing, \[
\mathbb{P}(T \geq 1.5) \leq \mathbb{P}(T + W \geq 1.5) \leq (1/2)\mathbb{E}[T] / 1.5 = 0.1666\ldots.
\]

A similar result holds for Chebyshev style smoothing applied to
\(X - \mathbb{E}(X)\) and \(-(X - \mathbb{E}(X))\). This leads to the
following theorem:

\begin{lemma}

Let \(a \geq 0\). Let \(X\) be a continuous nonnegative random variable
with decreasing density over the interval \([(1 / 2)a, (3 / 2)a]\) and
increasing density over \([-(3 / 2)a, -(1 / 2)a]\). Then \[
\mathbb{P}(|X - \mathbb{E}(X)| \geq a) \leq \frac{1}{2} \frac{\mathbb{V}(X)}{a^2}.
\]

\end{lemma}

For instance, consider \(Y \sim \textsf{Unif}([1, 10])\). Then
\(\mathbb{P}(Y \geq 9) = 1 / 9\). On the other hand \[
\mathbb{P}(Y \geq 9) = \mathbb{P}(Y - 5.5 \geq 3.5) = (1/2) \mathbb{P}(|Y - 5.5| \geq 3.5).
\]

The density of \(Y\) is constant (hence decreasing and increasing) over
\([1, 10]\), hence \[
\mathbb{P}(Y \geq 9) = \leq (1/2)(1/2) \mathbb{V}(Y) / 3.5^2 = 0.1377\ldots,
\] which is remarkably close to \(0.111\ldots\).

This immediately raises two questions.

\begin{enumerate}
\def\labelenumi{\arabic{enumi}.}
\item
  Are these new inequalities tight for these restricted classes of
  distributions?
\item
  Do these inequalities extend to discrete random variables?
\end{enumerate}

The purpose of this work is to answer both of these questions. For the
first question.

\begin{thm}

For any \(a\) and \(\mu\) greater than 0, there exists a continuous
nonnegative random variable \(X\) with decreasing density and
\(\mathbb{E}(X) = \mu\) such that \(\mathbb{P}(X \geq a)\) is
arbitrarily close to \(\mu / (2a)\).

\end{thm}

For the second question, the results are similar to the continuous case.
When \(X\) is a discrete integer valued random variable, then the
density of \(X\) (with respect to counting measure) is
\(\mathbb{P}(X = i)\) for \(i\) an integer.

\begin{thm}

Let \(X\) be a discrete random variable over \(\{0, 1, 2, \ldots\}\)
with finite expectation and decreasing density over the nonnegative
integers. That is, for all \(i \in \{0, 1, 2, \ldots\}\),
\(\mathbb{P}(X = i + 1) \leq \mathbb{P}(X = i)\).. Then for all
\(a \in \{1, 2, \ldots\}\), \[
\mathbb{P}(X \geq a) \leq \frac{\mathbb{E}(X)}{2a - 1}.
\] This inequality is tight in the sense that for each \(a\), there
exists a random variable \(X\) over the nonnegative integers with
unimodular density where
\(\mathbb{P}(X \geq a) = \mathbb{E}(X) / (2a - 1).\)

\end{thm}

\begin{thm}

Let \(X\) be a discrete integer valued random variable with finite
second moment and uniomodal density. That is, for all
\(i \in \{\ldots, - 1, 0, 1, \ldots\}\), there exists \(m\) such that
for all \(i \geq m\) it holds that
\(\mathbb{P}(X = i + 1) \leq \mathbb{P}(X = i)\), and for all
\(i \leq m\) it holds that
\(\mathbb{P}(X = i - 1) \leq \mathbb{P}(X = i)\). Then for all
\(a \in \{1, 2, \ldots\}\), \[
\mathbb{P}(|X - \mathbb{E}[X]| \geq a) \leq \frac{\mathbb{V}(X) + 1 / 12}{2a - 1}.
\] This inequality is not tight.

\end{thm}

\hypertarget{proof-of-theorem-1}{%
\section{Proof of Theorem 1}\label{proof-of-theorem-1}}

The example that shows that Markov's inequality for decreasing density
random variables is tight consists of a mixture of continuous uniforms.

\begin{proof}

Fix \(a > 0\). Let \(\epsilon \in (0, a)\), and let \(X\) be a mixture
of a uniform over \([0, \epsilon]\) and over \([0, 2a]\). Set
\(U_1 \sim \textsf{Unif}([0, \epsilon])\), and independently
\(U_2 \sim \textsf{Unif}([0, 2a])\). Write \(B \sim \textsf{Bern}(p)\)
to indicate a Bernoulli random variable where \(\mathbb{P}(B = 1) = p\)
and \(\mathbb{P}(B = 0) = 1 - p\). For such \(B\), let \[
X = (1 - B)U_1 + B U_2.
\] Then \(\mathbb{E}[X] = (1 - p)\epsilon / 2 + ap,\) so \[
p = \frac{\mathbb{E}[X] - \epsilon / 2}{a - \epsilon / 2}.
\]

Because \(\epsilon < a\), for \(X \geq a\) it must hold that \(B = 1\)
and \(U_2 \geq a\). Hence \[
\mathbb{P}(X \geq a) = \frac{1}{2}p = \frac{\mathbb{E}[X] - \epsilon / 2}{2a - \epsilon}.
\] As \(\epsilon \rightarrow 0\), this become arbitrarily close to the
inequality \(\mathbb{P}(X \geq a) \leq \mathbb{E}[X] / (2a).\)

\end{proof}

\hypertarget{proof-of-theorem-2}{%
\section{Proof of Theorem 2}\label{proof-of-theorem-2}}

Turning to the discrete problem, consider the space of all probability
sequences that have the same mean \(\mu\). That is, let \(\Omega\) be
the set of decreasing nonnegative sequences
\(p = (p_0, p_1, p_2, \ldots)\) such that
\(\sum_{i = 0}^\infty p_i = 1\) and \(\sum_{i = 0}^\infty i p_i = \mu\).
Then \(\Omega\) is compact in \(L_1\) and
\(f(p) = \sum_{i = a}^\infty p_i\) is a continuous function of \(p\).
Hence there exists a probability density \(p^*\) which maximizes
\(f(p)\).

The goal of this section is to learn about the maximizer by showing that
there exists a maximizer with certain properties. The outline is as
follows.

\begin{enumerate}
\def\labelenumi{\arabic{enumi})}
\tightlist
\item
  There must be a maximizer \(p^*\) such that \(p^*_i\) are equal for
  all \(i \in \{1, \ldots, a\}\).
\item
  There is a maximizer with at most two jumps in \(p^*_i\) for
  \(i \geq a\).
\item
  There is a maximizer that is a mixture of two uniforms.
\item
  Any maximizing mixture of two uniforms has to have
  \(\mathbb{P}(X \geq a) \leq \mathbb{E}[X] / (2a - 1)\).
\end{enumerate}

Start with the first property.

\begin{lemma}

There exists a maximizer \(p^* \in \Omega\) of \(f\) such that
\(p^*_1 = p^*_2 = \cdots = p^*_a\).

\end{lemma}

\begin{proof}

To show that there exists a maximizer where the \(p^*_i\) are equal for
all \(i \in \{1, \ldots, a\}\), consider a sequence where this is not
true. Then it is possible to push some probability towards 0 and towards
\(a\) in such a way that the expected value is preserved, and \(f\)
either stays the same or increases.

Let \(p\) be a probability sequence where there exists
\(i \in \{1, \ldots, a\}\) such that \(p_{i + 1} < p_i\). Let
\(g = p_i - p_{i + 1}\) be the size of the jump, and consider the
probability sequence \(p'\) defined as follows. \begin{align*}
p'_0 &= p_0 + g \frac{i}{i + 2}, \\
p'_j &= p_j - g\frac{2}{i + 2} & & \text{for all } j \in \{1, \ldots, i\}, \\
p'_{i + 1} &= p_{i + 1} + g \frac{i}{i + 2}, \\
p'_k &= p_k & & \text{for all } k \geq i + 2.
\end{align*}

Note \[
\sum_{i = 0}^\infty p'_j = g\left[\frac{i}{i + 2} - \frac{2i}{i + 2} + \frac{i}{i + 2}\right] + \sum_{i = 0}^\infty p_i = 1,
\] so this is a probability distribution. Also, with these changes: \[
p'_{i} = p_i - g \frac{2}{i + 2},
\] and \[
p'_{i + 1} = p_{i + 1} + g \frac{i}{i + 2} = p_{i} - g + g \frac{i}{i + 2} = p_i - g \frac{2}{i + 2}.
\] Hence \(p'_i = p'_{i + 1}\). Moreover, anywhere \(p_j = p_{j + 1}\)
for \(j \in \{1, \ldots, i - 1\}\), it still holds that
\(p'_j = p'_{j + 1}\). Also, \(f(p') \geq f(p)\). Finally, for \(X\)
with density \(p\) and \(X'\) with density \(p'\): \[
\mathbb{E}[X'] - \mathbb{E}[X] = (i + 1) g \frac{i}{i + 2} - \sum_{\ell = 1}^i \ell g \frac{2}{i + 2} = 0,
\] so they have the same mean.

Therefore, one application of the procedure gives a new random variable
with the same mean, the same or higher tail probability, and one fewer
jump in the probabilities at or before \(a\). So after at most \(a\)
applications of this procedure, there is a probability distribution
\(q\) with all values equal for \(i \in \{1, \ldots, a\}\),
\(f(q) \geq f(p)\), and the same mean as the original random variable,
and the same or greater tail probability.

\end{proof}

The next step is to show that there are at most two jumps in the density
beyond \(a\). To accomplish this, it helps to write the random variable
\(X\) as a mixture of discrete uniform random variables. This is only
possible because \(X\) has a decreasing density.

\begin{lemma}

Let \(X \in \{0, 1, \ldots\}\) have \(p_i = \mathbb{P}(X = i)\) be a
decreasing sequence. Set \(d_i = (i + 1)(p_i - p_{i + 1})\) and suppose
\(\mathbb{P}(D = i) = d_i\). Let
\([Y | D] \sim \textsf{Unif}(\{0, 1, \ldots, D\})\). Then \[
X \sim Y
\] and \(\mathbb{E}[X] = \mathbb{E}[D] / 2.\)

\end{lemma}

\begin{proof}

Since the \(p_i\) are decreasing, the \(d_i\) are nonnegative. Note that
\[
\sum_{i = 0}^\infty d_i = (p_0 - p_1) + 2(p_1 - p_2) + 3(p_2 - p_3) + \cdots,
\] which telescopes to give \[
\sum_{i = 0}^\infty d_i = p_0 + p_1 + p_2 + \cdots = 1,
\] so it is a probability distribution.

Now let \(\mathbb{P}(D = i) = d_i\) and
\([Y \mid D] \sim \textsf{Unif}(\{0, 1, \ldots, D\})\). Then
\(Y \in \{0, 1, \ldots\}\) and \begin{align*}
\mathbb{P}(Y = i) &= \sum_{j = i}^\infty \mathbb{P}(D = j) / (j + 1) \\
&= \sum_{j = i}^\infty (j + 1)(p_j - p_{j + 1})/ (j + 1) \\
&= p_i.
\end{align*}

Therefore \(Y \sim X\), and
\(\mathbb{E}[X] = \mathbb{E}[\mathbb{E}[Y \mid D]] = \mathbb{E}[D / 2] = \mathbb{E}[D] / 2.\)

\end{proof}

\begin{lemma}

There exists a maximizer \(p^* \in \Omega\) of \(f\) such that for \(X\)
with density \(p^*\), there exists \(i \geq a\) such that \(X\) is a
mixture of a point mass at 0 together with a uniform over
\(\{0, 1, \ldots, i\}\) and a uniform over \(\{0, 1, \ldots, i + 1\}\).

\end{lemma}

\begin{proof}

As in the previous lemma, let \(d_i = (i + 1)(p_i - p_{i + 1})\) so that
for \(D \sim d_i\) and
\([Y \mid D] \sim \textsf{Unif}(\{0, 1, \ldots, D\})\), it holds that
\(X \sim Y\).

Suppose there exists \(i\) and \(j\) with \(a \leq i < i + 2 \leq j\),
\(d_i > 0\), and \(d_j > 0\). Then create \(d'\) by pushing these two
probabilities towards each other. Let \(m = \min(d_i, d_j)\).
\begin{align*}
d'_{i}     &= d_i - m \\
d'_{i + 1} &= d_{i + 1} + m \\
d'_{j - 1} &= d_{j - 1} + m \\
d'_{j}     &= d_j - m.
\end{align*} Let \(D'\) have density \(d'\), and
\([Y' | D'] \sim \textsf{Unif}(\{0, 1, \ldots, D'\}\). Note \[
\mathbb{E}[D'] = \mathbb{E}[D] - m(i / 2) + m(i + 1) / 2 + m(j - 1) / 2 - m(j / 2) = \mathbb{E}[D].
\] So \(\mathbb{E}[Y'] = \mathbb{E}[Y] = \mathbb{E}[X]\).

The difference between \(\mathbb{P}(Y' \geq a)\) and
\(\mathbb{P}(Y \geq a)\) is related solely to the movement of
probability from \(i\) to \(i + 1\) in \(D'\) and \(j\) to \(j - 1\) in
\(D'\). Together, this gives \begin{align*}
\mathbb{P}(Y' \geq a) &= \mathbb{P}(Y \geq a) - m\frac{i - (a - 1)}{i + 1} + m \frac{i + 1 - (a - 1)}{i + 2} + m \frac{j - 1 - (a - 1)}{j} - m \frac{j - (a - 1)}{j + 1} \\
&= \mathbb{P}(Y \geq a) + m\left[\frac{a}{(i + 1)(i + 2)} - \frac{a}{j(j + 1)}\right]
\end{align*} which is strictly greater than \(\mathbb{P}(Y \geq a)\) for
\(i + 2 \leq j\).

Hence any maximizer must not have such a \(d_i\) and \(d_j\), which
completes the proof.

\end{proof}

Now to reduce it to only two uniforms.

\begin{lemma}

There exists a maximizer \(p^* \in \Omega\) of \(f\) such that for \(X\)
with density \(p^*\), there exits \(i \geq a\) such that \(X\) is a
mixture of a uniform over \(\{0, 1, \ldots, i\}\) with either 1) a point
mass at 0, or 2) a uniform over \(\{0, \ldots, i + 1\}\).

\end{lemma}

\begin{proof}

By the last lemma there exists an \(X\) that is the mixture of all
three. If for this \(X\), any one of \(d_0, d_i, d_{i + 1}\) is 0, then
the result is immediate. So suppose that all three are positive.

Let \(m = \min (d_0 i, d_{i + 1})\). Then set \begin{align*}
d'_0       &= d_0 - m / i \\
d'_i       &= d_i + m(1 + 1 / i) \\
d'_{i + 1} &= d_{i + 1} - m.
\end{align*}

Then \(d'_0 + d'_i + d'_{i + 1} = 1\) and all are still nonnegative.
Moreover, if \(D\) has density \(d\) and \(D'\) has density \(d'\), \[
\mathbb{E}[D'] = \mathbb{E}[D] - (0)(m / i) + i (m)(1 + 1 / i) - (i + 1)(m) = \mathbb{E}[D].
\]

Also, \[
\mathbb{P}(D' \geq a) = \mathbb{P}(D \geq a) + m(1 + 1 / i) - m > \mathbb{P}(D \geq a),
\] Hence the original \(X\) could not be a maximizer of \(f\), and any
maximizer must have at most two of \(\{d_0, d_i, d_{i + 1}\}\) positive.

\end{proof}

\begin{lemma}

For \(X\) a mixture of a point mass at 0 and uniform over
\(\{0, 1, \ldots, i\}\) with mean \(\mu\), for \(a\) a positive integer,
\[
\mathbb{P}(X \leq a) \leq \frac{\mu}{2a - 1}.
\]

\end{lemma}

\begin{proof}

Suppose for \(X\), \(d_0\) and \(d_i\) are positive, and all other
\(d_j = 0\). Then \(d_i = 2 \mu / i\), and for a given \(i\), \[
\mathbb{P}_i(X \geq a) = \frac{2 \mu}{i} \cdot \frac{i - (a - 1)}{i + 1}.
\] To see what value of \(i\) makes this as large as possible, consider
\[
r(i) = \frac{\mathbb{P}_{i + 1}(X \geq a)}{\mathbb{P}_{i}(X \geq a)} = \frac{i + 1 - (a - 1)}{i + 2} \cdot \frac{i}{i - (a - 1)}.
\] Then \[
\operatorname{sgn}(r(i) - 1) = \operatorname{sgn}(i - 2(a - 1)).
\]

Hence for integer \(a\) there is a maximum of \(\mathbb{P}_i(X \leq a)\)
at \(i = 2a - 2\) and \(i = 2a - 1\). In either case the maximum value
is \[
\mathbb{P}_{2a - 1}(X \geq a) = \mathbb{P}_{2a - 2}(X \geq a) = \frac{2\mu(a - 1)}{2(a - 1)[2(a - 1) + 1]} = \frac{\mu}{2a - 1}.
\]

\end{proof}

This leads to our last lemma concerning Markov's inequality for discrete
random varables.

\begin{lemma}

The maximum value of \(f(p)\) over \(p \in \Omega\) is \[
\frac{\mu}{2a - 1}.
\]

\end{lemma}

\begin{proof}

Earlier it was shown that there exists a maximizer that is the mixture
of at most two \(X_1\) and \(X_2\) which are themselves a mixture of at
most a point mass at zero and a uniform over \(\{0, 1, \ldots, i\}\) for
some \(i\). Let \(X\) equal \(X_1\) with probability \(p_1\), and
\(X_2\) with probability \(p_2\). Then from the previous lemma, our goal
is to maximize \[
\mathbb{P}(X \geq a) = p_1 \mathbb{P}(X_1 \geq a) + p_2 \mathbb{P}(X_2 \geq a) \leq p_1 \frac{\mu_1}{2a - 1} + p_2 \frac{\mu_2}{2a - 1}.
\] subject to \begin{align*}
p_1 \mu + p_2 \mu &= \mu \\
p_1 + p_2 &= 1 \\
p_1, p_2 &\geq 0.
\end{align*} This is a continuous function over a compact space, so the
maximum either occurs at a boundary (where one of \(p_1\) or \(p_2\) is
0) or at a Lagrange critical point. These occur for choice of
\(\lambda_1\) and \(\lambda_2\) such that \begin{align*}
0 &= \nabla(p_1 \mu_1 / (2a - 1) + p_2 \mu_2 / (2a - 1)) + \lambda_1 \nabla(p_1 \mu_1 + p_2 \mu_2) + \lambda_2 \nabla(p_1 + p-2) \\
&= (\mu_1, \mu_2)(\lambda_1 + 1 / (2a - 1)) + (1, 1)\lambda_2 = 0.
\end{align*}

Hence there are only critical points if \(\mu_1 = \mu_2\). In this case,
\(X_1 \sim X_2\) and so the mixture is unneeded. For
\(X_1 \not\sim X_2\), the optimal value occurs at \(p_1 = 0\) or
\(p_2 = 0\) and the mixture resolves down to a single distribution.

The previous lemma then gives the bound.

\end{proof}

This proves Theorem 2.

Applied to our initial example where
\(U \sim \textsf{Unif}(\{0, 1, 2, \ldots, 9, 10\})\), this gives \[
\mathbb{P}(U \geq 9) \leq \frac{5}{2(9) - 1} = \frac{5}{17} \approx 0.2941\ldots,
\] much closer to the true answer of \(0.1818\ldots\).

\hypertarget{proof-of-theorem-3}{%
\section{Proof of Theorem 3}\label{proof-of-theorem-3}}

\begin{proof}

Let \(X\) be a discrete integer-valued random variable with finite
second moment, and let \(U \sim \textsf{Unif}([-0.5, 0.5])\) be
independent of \(X\). Then \(X + U\) is a continuous random variable
with unimodal density. For \(X \geq a\), \(X + U \geq a - 1 / 2\) and
for \(X \leq a\), \(X+ U \leq a + 1 / 2\) suffices. Hence \begin{align*}
\mathbb{P}(|X - \mathbb{E}(X)| \geq a) &= \mathbb{P}(|X + U - \mathbb{E}(X)| \geq a - 1 / 2) \\
   &\leq \frac{1}{2}\cdot \frac{\mathbb{V}(X + U)}{(a - 1 / 2)^2} \\
   &\leq \frac{1}{2}\cdot \frac{\mathbb{V}(X) + \mathbb{V}(U)}{(a - 1 / 2)^2} \\
   &\leq \frac{1}{2}\cdot \frac{\mathbb{V}(X) + 1 / 12}{(a - 1 / 2)^2}.
\end{align*}

\end{proof}

\hypertarget{refs}{}
\leavevmode\hypertarget{ref-huber2019f}{}%
Huber, M. 2019. ``Halving the Bounds for the Markov, Chebyshev, and
Chernoff Inequalities Through Smoothing.'' \emph{American Mathematical
Monthly} 126 (10): 915--27.

\leavevmode\hypertarget{ref-markov1884}{}%
Markov, A. 1884. ``On Certain Applications of Algebraic Continued
Fractions.'' PhD thesis, St. Petersburg University.

\leavevmode\hypertarget{ref-motwanir1995}{}%
Motwani, R., and P. Raghavan. 1995. \emph{Randomized Algorithms}.
Cambridge Univ. Press.

\leavevmode\hypertarget{ref-tchebichef1867}{}%
Tchebichef, P. 1867. ``Des Valeurs Moyennes.'' \emph{Journal de
Mathématiques Pures et Appliquées} 2 (12): 177--84.

\leavevmode\hypertarget{ref-williamson2011}{}%
Williamson, D. P., and D. B. Shmoys. 2011. \emph{The Design of
Approximation Algorithms}. Cambridge University Press.

\end{document}